\documentclass[11pt]{article}
 \usepackage{epsfig}
 \usepackage{amssymb,amsmath,amsthm,amscd}
 \usepackage{latexsym}

\def\R{\mathbb{R}}

\newcommand{\be}{\begin{equation}}
\newcommand{\ee}{\end{equation}}

\newtheorem{thm}{Theorem}[section]

\newtheorem{cor}[thm]{Corollary}
\newtheorem{lem}[thm]{Lemma}

\theoremstyle{definition}

\theoremstyle{remark}

\newcounter{labelflag} 

\newcommand{\Label}[1]{
                       \ifnum\thelabelflag=1
                          \ifmmode
                             \makebox[0in][l]{\qquad\fbox{\rm#1}}
                          \else
                             \marginpar{\vspace{0.7\baselineskip}
                                        \hspace{-1.1\textwidth}
                                        \fbox{\rm#1}}
                          \fi
                       \fi
                       \label{#1}
                      }

\begin{document}

\baselineskip=1.3\baselineskip

\begin{titlepage}
\title{\Large\bf   Almost Periodic Dynamics of Perturbed Infinite-Dimensional Dynamical Systems}
\vspace{7mm}

\author{
  Bixiang Wang  \thanks {Supported in part by NSF  grant DMS-0703521}
\vspace{1mm}\\
Department of Mathematics, New Mexico Institute of Mining and
Technology \vspace{1mm}\\ Socorro,  NM~87801, USA \vspace{1mm}\\
Email:    bwang@nmt.edu}
\date{}
\end{titlepage}

\maketitle

\noindent
{\bf Abstract}.
 This paper is  concerned with the dynamics of an infinite-dimensional gradient system
under small almost periodic perturbations. Under the assumption
 that the original autonomous system has a
global attractor given as the union of unstable manifolds of a finite number of
hyperbolic equilibrium solutions, we prove that
the perturbed non-autonomous system has exactly the same number of almost periodic solutions.
As a consequence,
  the pullback attractor of the perturbed system is given by the union of unstable
manifolds of  these finitely many almost periodic solutions.
   An application of the result to the  Chafee-Infante
     equation is discussed.
  \vspace{6mm}


 \section{Introduction}
\setcounter{equation}{0}

In this paper, we study the dynamics of an infinite-dimensional gradient system
under almost periodic perturbations. Let $A_0$ be a sectorial operator in a Banach
space $X$ and  $X^\alpha$ be the fractional powers of $X$  with
 $0\le \alpha <1$.
Consider the autonomous nonlinear equation
\be \label{intro1}
{\frac {dx}{dt}} +A_0 x = f(x), \quad \ x(0) = x_0,
\ee
where $f: X^\alpha \to  X$ is  locally Lipschitz continuous.
   Suppose for each $x_0 \in X^\alpha$,
the initial-value problem \eqref{intro1}
has a unique solution $x \in C([0, \infty); X^\alpha)$.
Assume further that equation \eqref{intro1}  has a Liapunov function
and only a finite number of equilibrium solutions $x^*_i$, $1\le i\le n$.
If  problem \eqref{intro1} has a global attractor $\cal{A}$,  then it is well
know  (see, e,g., \cite{hal1} )  that  $\cal{A}$ is given by  the union of unstable manifolds
of  all the equilibrium solutions, i.e.,
\be
\label{att}
{\cal{A}} = \bigcup_{i=1}^n W^u (x_i^*),
\ee
where $W^u (x^*_i)$ is the unstable manifold of $x_i^*$.

In this paper,  we  want to explore the effect of almost periodic perturbations
on the structure of the attractor $\cal{A}$.
More precisely, given $\epsilon >0$, consider  the non-autonomously perturbed
equation
\be
\label{intro2}
{\frac {dx}{dt}} +A_0 x = f(x) + g_\epsilon (t,x), \quad \ x(\tau) = x_0,
\ee
where $g_\epsilon (t,x): \R \times X^\alpha
\to  X$ is an almost periodic function
in $t$ uniformly in $x$, and $g_\epsilon  \to 0 $
in a sense (which will be made clear  in the next section).
Under certain conditions, we will prove that, for every small $\epsilon>0$,
problem \eqref{intro2}     has
exactly
 $n$ almost periodic solutions
$\phi^*_{i, \epsilon}$, $i\le i \le n$, and each  $\phi^*_{i, \epsilon}$
corresponds  to an equilibrium solution $x^*_i$  of problem \eqref{intro1}.
This result along with  \cite{car1}  implies  that the pullback attractor
$\{{\cal{A}}_\epsilon (t)\}_{t\in \R}$  of problem \eqref{intro2}
can be characterized by
\be\label{attp}
{\cal{A}}_\epsilon (t) =\bigcup_{i=1}^n W^u_\epsilon (\phi^*_{i, \epsilon} ) (t),
\qquad t \in \R,
\ee
where
$W^u_\epsilon (\phi^*_{i, \epsilon} ) $
is the unstable manifold of the almost periodic solution
$\phi^*_{i,\epsilon}$,
which will be defined in Section 2.
By \eqref{att} and \eqref{attp} we see that
the structure of the attractor $\cal{A}$ is preserved under a small
almost periodic perturbation, and the almost periodic solutions
$\phi^*_{i,\epsilon}$
play  exactly the  same role  on the dynamics  of the perturbed system
as  the equilibrium solutions $x^*_i$ do on the autonomous equation.
In this sense, almost periodic solutions are appropriate extension
of equilibrium solutions to almost periodic systems.
In particular, when the perturbation $g_\epsilon$ is time-periodic, it can be
proved that the almost periodic solutions $\phi^*_{i, \epsilon}$
are actually  periodic solutions. In this case, the dynamics of the corresponding
periodic systems  is  completely governed by  a set of  finitely many periodic solutions.

Closely  related to the present paper, a more general non-autonomous perturbation
problem was recently  studied by Carvalho et. al. in \cite{car1, car2}, where the
perturbation $g_\epsilon$  was not assumed to be almost periodic  in time.
In that case, the authors  proved that,  for each  $i$ with $1\le i \le n$
and each small $\epsilon>0$,  the equation \eqref{intro2} has a complete
solution $\xi^*_{i, \epsilon}$ (i.e., $\xi^*_{i, \epsilon}$ is defined for all $t \in \R$)
which corresponds to  the equilibrium solution $x^*_i$ of equation \eqref{intro1}.
Further, the authors proved that
the pullback attractor
$\{{\cal{A}}_\epsilon (t)\}_{t\in \R}$ of the perturbed equation  is given  by
\be\label{attp2}
{\cal{A}}_\epsilon (t) =\bigcup_{i=1}^n W^u_\epsilon (\xi^*_{i, \epsilon} ) (t),
\qquad t \in \R,
\ee
where
$W^u_\epsilon (\xi^*_{i, \epsilon} ) $
is the unstable manifold of  $\xi^*_{i,\epsilon}$.
This is a  remarkable  result on the exact structure
of pullback attractors of  an   infinite-dimensional
non-autonomous  system.
Note that, in the nontrivial case, the  non-autonomous equation
\eqref{intro2} has infinitely many complete solutions. In fact,
every  solution   starting from  the attractor  is a complete solution.
The authors of \cite{car1} successfully identified a finite number of
complete solutions which are crucial in determining the structure of the
pullback attractor.
In this paper,  we will prove that the complete solution
$\xi^*_{i, \epsilon}$ ($1\le i\le n$)  founded  in \cite{car1}
are actually almost periodic solutions when the perturbation
$g_\epsilon$ is uniformly almost periodic.
We further demonstrate that  problem \eqref{intro2}  has no other
almost periodic solutions except those  $\xi^*_{i, \epsilon}$.
The convergence of solutions  of the perturbed equation was also established
in  \cite{car1}. Particularly,  they proved that every solution
of the equation converges to one of the complete bounded
solutions $ \xi^*_{i, \epsilon}$ as $t \to \infty$.

The solutions and dynamics of almost periodic  differential equations
have long
been  investigated in the literature, see, e.g.,
\cite{bal1, che1, fin1,  sel1,  she1, she2, she3, she4, vui1,  war1, yos1, zai1}.
Particularly, the $\omega$-limit sets of almost periodic equations have
been studied by  the authors of \cite{sel1, she1, she2, she3, she4}.

In the next section, we will recall some
concepts on sectorial operators and uniformly almost periodic
functions. We will present our main results in this section.
Section 3 is devoted to  the  proof of the results. We first show the
existence of almost periodic solutions  for the non-autonomously perturbed
equation, and then  prove the number of  almost periodic solutions is finite.
In the last section, we discuss an application of our results  to
the Chafee-Infante equation under almost periodic perturbations.

\section{Notation and Main Results}
\setcounter{equation}{0}

Let $A_0$ be a sectorial operator  in a
  Banach space  $X$ with norm $\| \cdot \|$. Suppose there is
  a number $a\in \R$ such that the fractional powers
  $ (A_0+ a I)^\alpha$ are well defined for all
  $\alpha \in \R$.  Set $A_1 = A_0 + aI$  and
  $X^\alpha = D(A^\alpha_1)$,  the domain of $A_1^\alpha$,
  for  $\alpha \ge 0$.
  Then it is known (see, e.g., \cite{hen1})
   that  $X^\alpha$ is a Banach space with
  the graph norm  $\|x\|_\alpha = \| A_1^\alpha \|$
  for $x \in D(A_1^\alpha)$.
  From now on, we assume
  $\alpha \in [0, 1)$ and $f: X^\alpha \to  X$ is a
  continuously
  differentiable function that is Lipschitz  continuous in bounded
  subsets of $X^\alpha$.  Consider the   nonlinear   autonomous
  equation
  \be \label{autoeq}
{\frac {dx}{dt}} +A_0 x = f(x), \quad  t>0 \qquad \mbox{and} \quad \ x(0) = x_0 \in X^\alpha .
\ee
Suppose Problem \eqref{autoeq} is  well-posed in $X^\alpha$, that is,
for each $x_0 \in X^\alpha$ the initial-value problem has a unique solution
$x \in C([0,\infty); X^\alpha)$ which depends continuously  on initial data
$x_0$ in $X^\alpha$. Let
$\{S(t)\}_{t \ge 0}$ be the evolution semigroup  generated by problem
\eqref{autoeq}, that is, for every $t\ge 0$ and $x_0 \in X^\alpha$,
$S(t)x_0 = x(t, x_0)$,   the solution of equation \eqref{autoeq}
 at time $t$ with initial condition  $x_0$.
 If $x_0^* \in X^\alpha$   and $S(t)x_0^* = x_0^*$ for all $t\ge 0$, then
 $x_0^*$ is called an equilibrium solution  of
$\{S(t)\}_{t \ge 0}$.  An  equilibrium solution $x_0^*$
is said to be hyperbolic  if   the spectrum of $A= A_0 - f^\prime (x_0^*)$
does not intersect  the imaginary axis, and  there is a projection
$P: X \to  X$ such that the following conditions are satisfied,
for some positive numbers $M$ and $\beta$:
\be\label{dichp1}
e^{-At} P = P e^{-At}, \quad \forall\  t \ge 0.
\ee
 \be\label{dichp2}
 \| e^{-At} P x \| \le Me^{\beta t}\| x\|, \quad  \forall \ x\in X, \  \forall \ t \le 0.
 \ee
  \be\label{dichp3}
  \| e^{-At} (I-P) x \| \le Me^{-\beta t}\| x\|, \quad  \forall \ x\in X, \  \forall \ t \ge 0.
  \ee
   Then it follows from  Lemma 7.6.2 in \cite{hen1} that there is a positive number
  $M_1$ such that for all $x \in X$:
  \be
  \label{dichp}
  \| e^{-A(t-s)} Px \|_\alpha \le M_1 e^{\beta (t-s)} \|x\|,
  \quad \forall \  t < s,
  \ee
  \be
  \label{dichq}
  \| e^{-A(t-s)} (I-P)x \|_\alpha \le M_1 e^{-\beta (t-s)}
  \max \{1, (t-s)^{-\alpha}  \} \|x\|,
  \quad \forall \  t > s.
  \ee

  Throughout this paper, we assume that $\{S(t)\}_{t \in \R}$ has the following properties:

  (H1)  \  $\{S(t)\}_{t \in \R}$ has a global attractor $\cal{A}$ in $X^\alpha$.

  (H2) \   $\{S(t)\}_{t \in \R}$ has only a finite number of hyperbolic equilibrium solutions
  $x^*_i$, $1\le i \le n$.

  (H3)  \  $\{S(t)\}_{t \in \R}$ has a Liapunov function in $X^\alpha$.

  Under these conditions
   it is well known (see, e.g., \cite{hal1})  that
  the attractor $\cal{A}$ has the structure
  \be
  \label{autoatt}
  {\cal{A}} = \bigcup_{i=1}^n W^u (x^*_i),
  \ee
  where  $W^u (x^*_i)$ is the unstable manifold of $x^*_i$.
   We intend to examine the behavior of this attractor under
   small almost periodic perturbations.
   Recall that a continuous function
   $g: \R \to X$ is called  an almost periodic function
   if for every $\delta>0$
    there exists a positive
   number $l$ (depending on $\delta$)
   such that every interval $I$ of length $l$ contains a number $s$ for which
   $ \| g(t+s) - g(t) \| < \delta \quad \mbox{for all} \  t \in \R.$
   Every almost periodic function  $g$ is bounded, i.e., $g \in C_b(\R, X)$,
   where   $ C_b(\R, X)$   is the Banach space of all continuous and bounded functions
   from $\R$ to $X$ with norm $\|g\|_{C_b(\R, X)}$
   $=\sup\limits_{t \in \R} \| g(t) \|$.
   A   continuous function
   $g(t,x): \R \times X^\alpha \to X$
   is said to be almost periodic in $t$ uniformly in $x$ if for every $\delta>0$
   and every compact subset  $K$ of $X^\alpha$ there exists a positive
   number $l$ (depending on $\delta$ and $K$)
   such that every interval $I$ of length $l$ contains a number $s$ for which
   $$ \| g(t+s, x) - g(t,x) \| < \delta \quad \mbox{for all} \  t \in \R
   \quad \mbox{and} \   x \in K.$$
   Given $\epsilon \in [0, 1)$, let $g_\epsilon (t,x): \R \times X^\alpha \to X$
   be almost periodic in $t$ uniformly in $x$. We further assume that $g_\epsilon$
   is continuously differentiable in $x \in X^\alpha$  and satisfies, for every $r>0$,

   (H4) \  $\lim\limits_{\epsilon \to 0} \sup\limits_{t \in \R} \sup\limits_{\|x\|_\alpha \le r}
   \left (
     \| g_\epsilon (t,x) \| + \| {\frac {\partial g_\epsilon}{\partial x}}(t,x) \|_{L(X^\alpha, X)}
   \right )
   =0.$

   Consider the equation with almost periodic perturbations
   \be
\label{nautoeq}
{\frac {dx}{dt}} +A_0 x = f(x) + g_\epsilon (t,x), \quad \ x(\tau) = x_0 \in X^\alpha,
\ee
where $t> \tau$ with $\tau \in \R$.
 Suppose,  for each $x_0 \in X^\alpha$,
  the initial-value problem  \eqref{nautoeq} has a unique solution
$x \in C([\tau,\infty); X^\alpha)$ which depends continuously  on initial data
$x_0$ in $X^\alpha$. Let
$\{S_\epsilon(t, \tau): \tau \in \R,  \ t\ge \tau\}$
 be the evolution process   generated by  the non-autonomous equation
\eqref{nautoeq}, that is, for every  $\tau \in \R$, $t\ge \tau$ and $x_0 \in X^\alpha$,
$S_\epsilon (t, \tau)x_0 = x_\epsilon (t, \tau, x_0)$,   the solution of equation \eqref{nautoeq}
 with  $x_\epsilon (\tau, \tau,x_0) =x_0$.
 For convenience, we  also  write  the process as
 $S_\epsilon(\cdot, \cdot)$ occasionally.

 A family of compact subsets of $X^\alpha$,
 $\{{\cal{A}}_\epsilon (t) \}_{t \in \R}$, is called a pullback attractor
 of  $S_\epsilon(\cdot, \cdot)$ if the following conditions are  fulfilled:

 (1) $\{{\cal{A}}_\epsilon (t) \}_{t \in \R}$   is invariant, i.e.,  for all $  \tau \in  \R$
 and $t \ge \tau$, \
 $S_\epsilon (t,\tau)  {\cal{A}}_\epsilon (\tau)   ={\cal{A}}_\epsilon (t).$

 (2) $\{{\cal{A}}_\epsilon (t) \}_{t \in \R}$ attracts all bounded subsets $B$ of  $X^\alpha$, i.e.,
 \  $dist (S_\epsilon (t,\tau)B, \ {\cal{A}}_\epsilon (t)) \to  0 $
 as $\tau \to -\infty$, where the Hausdorff semi-distance in $X^\alpha$ is used.

  Suppose $S_\epsilon (\cdot, \cdot)$ has a pullback attractor
  $\{{\cal{A}}_\epsilon (t) \}_{t \in \R}$.  We will characterize the structure of
  this attractor in the present paper.
  To the end, we further assume that    there are
  $\epsilon_0 >0$ and  a compact subset $K$ of $X^\alpha$ such that

  (H5) \ $\bigcup\limits_{\epsilon \le \epsilon_0} \bigcup\limits_{t\in \R} {\cal{A}}_\epsilon (t)
  \subseteq  K$.

  Under assumptions (H1)-(H5) we will prove that problem \eqref{nautoeq}
  has exactly $n$ almost periodic solutions and the pullback attractor
  is the union of unstable manifolds of all the almost periodic solutions.
  Our main  results are summarized  as follows.

  \begin{thm}\label{thm1}
  Suppose (H1)-(H5) hold and $g_\epsilon (t,x)$ is almost periodic in $t \in \R$ uniformly
  in $x \in X^\alpha$.  Then there  is $\epsilon_0 >0$ such that for every $\epsilon \in (0,\epsilon_0)$
   problem \eqref{nautoeq} has  exactly $n$ almost periodic solutions $\phi^*_{i, \epsilon}$,
   $1\le i \le n$.  Further, each  $\phi^*_{i, \epsilon}$ corresponds to the equilibrium solution
   $x^*_i$  of problem \eqref{autoeq} in the  sense:
   $$ \lim_{\epsilon \to 0} \sup_{t \in \R} \| \phi^*_{i, \epsilon}  (t) - x^*_i \|_\alpha =0,
   \quad \forall \ 1 \le i \le n.
   $$
     \end{thm}

  As a  consequence of this result,  it follows
  from Theorem 2.11  of  \cite{car1} that the pullback attractor
  of problem \eqref{nautoeq}  can be characterized  by the union of unstable manifolds
  of the almost periodic solutions.  More precisely,   we have:

  \begin{cor}\label{thm2}
  Suppose (H1)-(H5) hold and $g_\epsilon (t,x)$ is almost periodic in $t \in \R$ uniformly
  in $x \in X^\alpha$.  Then there  is $\epsilon_0 >0$ such that for every $\epsilon \in (0,\epsilon_0)$:

  (1) \  The pullback attractor
  $\{{\cal{A}}_\epsilon (\tau) \}_{\tau \in \R}$
  of  problem \eqref{nautoeq}  is given by
  $$
{\cal{A}}_\epsilon (\tau) =\bigcup_{i=1}^n W^u_\epsilon (\phi^*_{i, \epsilon} ) (\tau),
\qquad  \tau  \in \R,
$$
where   $ W^u_\epsilon (\phi^*_{i, \epsilon} ) (\tau)  $
consists of all $x_0 \in X^\alpha$ such that
there is a backwards   solution $x(t, \tau, x_0)$ of problem \eqref{nautoeq}
for which $x(\tau, \tau, x_0) =x_0$ and
$ \| x(t,\tau, x_0)   -   \phi^*_{i,\epsilon} (t)  \|_\alpha
\to 0$ as $t \to -\infty$.
In addition,   for every $\tau \in \R$, the Hausdorff  dimension of
${\cal{A}}_\epsilon (\tau)$ is the same as
that of  $ \cal{A}$.

(2) \  For every $\tau \in \R$ and $x_0 \in X^\alpha$
 there  exists $i \in \{1, \cdots, n\}$ such that
 $ \| S_\epsilon (t, \tau) x_0   -   \phi^*_{i, \epsilon}  (t) \|_\alpha
 \to 0 $  as $t \to \infty$.  In addition, if $S(t, \tau)x_0$
 is a complete bounded solution, then there is $j \neq i$ such that
 $ \| S_\epsilon (t, \tau) x_0   \to  \phi^*_{j, \epsilon} (t) \|_\alpha
 \to 0 $  as $t \to - \infty$.
  \end{cor}

\section{Proof   of  Main Results}
\setcounter{equation}{0}

This section is devoted to  the proof of our main results.
We first prove the existence of almost periodic solutions of
problem \eqref{nautoeq}.

\begin{lem}
\label{lem31}
Suppose (H4) holds and $g_\epsilon (t,x)$ is almost periodic in $t \in \R$ uniformly
  in $x \in X^\alpha$.
  Then for every hyperbolic  equilibrium solution $x_0^*$ of problem \eqref{autoeq},
  there  are positive numbers  $\delta_0 $  and
   $\epsilon_0$ such that  for each $\epsilon \in (0, \epsilon_0)$,  problem
  \eqref{nautoeq} has a unique  almost periodic solution
  $x_\epsilon^*: \R \to  X^\alpha$  which satisfies
  $\| x_\epsilon^*  -x_0^*\|_{C_b(\R, X^\alpha)}$
   $\le \delta_0$. Furthermore,
  $x_\epsilon^* \to x^*_0$ in $C_b(\R, X^\alpha)$ as
  $\epsilon \to 0$.
  \end{lem}

  \begin{proof}
  Suppose that  $x: \R \to X^\alpha$  is
   an almost periodic  solution of equation \eqref{nautoeq}. Then for
  $y =x-x_0^*$, $\tau \in \R$ and $t \ge \tau$  we have
  \be
  \label{p31_1}
  {\frac {dy}{dt}} + Ay = h(y) +g_\epsilon (t, x_0^* +y),
  \ee
  where $A= A_0 -f^\prime (x_0^*)$ and
  $h(y) =  f(y+ x_0^*) - f(x_0^*)- f^\prime (x_0^*) y$.
  Note that $h(0) = 0$ and $h^\prime (0) =0$.
  It follows from \eqref{p31_1} that, for all $t \ge \tau$ with $\tau \in \R$,
  \be
  \label{p31_2}
  y(t) = e^{-A(t-\tau)} y(\tau)
  + \int_\tau^t e^{-A(t-s)} \left (
  h(y(s))  + g_\epsilon (s, x_0^* + y(s))
  \right )ds.
  \ee
  Since $x_0^*$ is a  hyperbolic equilibrium solution of
  problem \eqref{autoeq}, there is a projection $P$ satisfying
  \eqref{dichp1}-\eqref{dichq}. Applying $P$
  and $I-P$ to \eqref{p31_2} we find that
   \be
  \label{p31_3}
  Py(t) = e^{-A(t-\tau)}  Py(\tau)
  + \int_\tau^t e^{-A(t-s)} P \left (
  h(y(s))  + g_\epsilon (s, x_0^* + y(s))
  \right )ds,
  \ee
  and
   \be
  \label{p31_4}
  (I-P)y(t) = e^{-A(t-\tau)}  (I-P)y(\tau)
  + \int_\tau^t e^{-A(t-s)} (I-P) \left (
  h(y(s))  + g_\epsilon (s, x_0^* + y(s))
  \right )ds.
  \ee
  Note that $y: \R \to X^\alpha$ is bounded since it is
  almost periodic. Letting $\tau \to +\infty$
  and $\tau \to - \infty$ in \eqref{p31_3} and
  \eqref{p31_4}, respectively,
  by \eqref{dichp2} and \eqref{dichp3} we obtain
  $$
  Py(t) =    -  \int_t^\infty e^{-A(t-s)} P \left (
  h(y(s))  + g_\epsilon (s, x_0^* + y(s))
  \right )ds,
$$
  and
$$
  (I-P)y(t) =      \int_{-\infty}^t e^{-A(t-s)} (I-P) \left (
  h(y(s))  + g_\epsilon (s, x_0^* + y(s))
  \right )ds.
 $$
 Therefore,  $y$ must satisfy the equation, for all $t \in \R$,
 \be
 \label{p31_6}
 y(t)=
 \int_{-\infty}^t e^{-A(t-s)} (I-P) \phi_\epsilon (s, y(s)) ds
 -
 \int_t^\infty e^{-A(t-s)} P \phi_\epsilon (s, y(s)) ds,
 \ee
 where $\phi_\epsilon (t,x) = h(x)  + g_\epsilon (t, x_0^* +x)$
 for $t \in \R$ and $x \in X^\alpha$.
 Conversely, if $y: \R \to X^\alpha$ is  almost periodic and fulfills
 \eqref{p31_6}, then $y$ is  an almost periodic solution of equation
 \eqref{nautoeq}. So finding an almost periodic solution
 of equation \eqref{nautoeq} amounts
 to finding a fixed point of the mapping $\cal{F}$ given by
 \be
 \label{mapF}
  {\cal{F}} (y)  (t)
  = \int_{-\infty}^t e^{-A(t-s)} (I-P) \phi_\epsilon (s, y(s)) ds
 -
 \int_t^\infty e^{-A(t-s)} P \phi_\epsilon (s, y(s)) ds.
 \ee
 Given $\delta>0$, set
 $$Z=\{y: \R \to X^\alpha, \ y \mbox{ is almost periodic  and } \
 \sup_{t\in \R} \| y(t) \|_\alpha \le \delta \}.
 $$
 Then $Z$ is a complete metric space with distance induced by
 the  norm  of $C_b(\R, X^\alpha)$.
 In what follows, we will prove,
 for a  sufficiently small $\delta$, $\cal{F}$ has a unique  fixed point
 in $Z$.

 Note  that $h$ is continuously   differentiable and $h^\prime (0) =0$.
 So there is $\delta_1>0$ such that for all $x \in X^\alpha$ with
 $\| x \|_\alpha < \delta_1$,
 \be
 \label{p31h1}
 \| h^\prime (x) \|_{L(X^\alpha, X)}
 < \min \left  \{
    {\frac \beta{8M_1}}, \ \ {\frac {\beta^{1-\alpha}}{8M_1 \Gamma(1-\alpha)}},
    \ \ {\frac 1{2M_1 (4\beta^{-1} + 2 \beta^{\alpha -1} \Gamma (1-\alpha) )} }
\right  \},
 \ee
 where $\beta$ and $M_1$ are the positive constants
 in \eqref{dichp} and \eqref{dichq}, and $\Gamma (\alpha)$ is the
 value of the $\Gamma$ function  at $\alpha$.
 Let $\delta_0   =  \min \{1, \  \delta_1\}$.
 Then given $\delta \in (0, \delta_0]$,  by (H4) we find that there is
 $\epsilon_0>0 $ depending on $\delta$ such that for all $\epsilon < \epsilon_0$,
 \be
 \label{p31g1}
 \sup_{t \in \R} \ \
 \sup_{ \| x \|_\alpha \le  1+ \|x_0^* \|_\alpha}
 \| g_\epsilon (t,x) \|
 <  \min  \left  \{
    {\frac {\beta\delta}{8M_1}}, \ \ {\frac {\beta^{1-\alpha}\delta}{8M_1 \Gamma(1-\alpha)}}
    \right  \},
    \ee
    and
     \be
 \label{p31g2}
 \sup_{t \in \R} \ \
 \sup_{ \| x \|_\alpha \le  1+ \|x_0^* \|_\alpha}
 \| {\frac {\partial g_\epsilon}{\partial x}} (t,x) \|_{L(X^\alpha, X)}
 <
  {\frac 1{2M_1 (4\beta^{-1} + 2 \beta^{\alpha -1} \Gamma (1-\alpha) )} }.
 \ee
 Given $y \in Z$ it follows from
 \eqref{dichp}-\eqref{dichq} and
 \eqref{mapF} that
 \be
 \label{p31_10}
 \| {\cal{F}} (y) (t) \|_\alpha
 \le M_1 \int_t ^\infty e^{\beta (t-s)} \| \phi_\epsilon (s, y(s))\| ds
 + M_1 \int_{-\infty}^t e^{-\beta (t-s) } (1+ (t-s)^{-\alpha})
 \| \phi_\epsilon (s, y(s)) \|ds.
 \ee
 The right-hand side of \eqref{p31_10} is estimated as follows.
 Since $y \in Z$   and $\delta<\delta_1$, by \eqref{p31h1} we get,
 for all $s \in \R$,
 $$
 \| h(y(s)) \|
 =\| h(y(s)) -h(0) \| \le \sup_{\| x\|_\alpha \le \delta} \| h^\prime (x ) \|_{L(X^\alpha, X)} \sup_{s \in \R} \| y(s) \|_\alpha
 \le \min
 \left  \{
    {\frac {\beta\delta}{8M_1}}, \ \ {\frac {\beta^{1-\alpha}\delta}{8M_1 \Gamma(1-\alpha)}}
    \right  \},
     $$
     which along with
     \eqref{p31g1} shows that, for all  $\epsilon < \epsilon_0$,
     \be
     \label{p31_11}
    \sup_{s \in \R}  \| \phi_\epsilon (s, y(s)) \|
     \le  \min \left  \{
    {\frac {\beta\delta}{4M_1}}, \ \ {\frac {\beta^{1-\alpha}\delta}{4M_1 \Gamma(1-\alpha)}}
    \right  \} .
        \ee
        By \eqref{p31_10} and \eqref{p31_11} we find that, for all $t \in \R$,
  $$
 \| {\cal{F}} (y) (t) \|_\alpha
 \le {\frac {2M_1}\beta} \sup_{s\in \R} \|\phi_\epsilon (s, y(s)) \|
 + M_1\sup_{s\in \R} \|\phi_\epsilon (s, y(s)) \|
    \int_{-\infty}^t e^{-\beta (t-s) }  (t-s)^{-\alpha}ds
    $$
    \be
    \label{p31_12}
       \le  {\frac 12} \delta + {\frac {\beta^{1-\alpha} }{4\Gamma(1-\alpha)} }\delta
    \int_0^\infty e^{-\beta s} s^{-\alpha} ds
    \le \delta,
    \ee
    where  we have used the integral
    $  \int_0^\infty e^{-\beta s} s^{-\alpha} ds
    = \beta^{\alpha -1} \Gamma (1-\alpha) $   for  $ 0\le \alpha <1.$
    Note that \eqref{p31_12}  implies  ${\cal{F}} (y)
    \in C_b(\R, X^\alpha)$  with norm $\|{\cal{F}}(y) \|_{C_b(\R, X^\alpha)}
    \le \delta$.   We now prove  ${\cal{F}}(y)$ is almost periodic.
   If  $y \in Z$, then $y: \R \to  X^\alpha$ is almost periodic  and hence
   the set $\{ y(t): t \in \R \}$ is precompact in $X^\alpha$.
   As a consequence  of this,
   it follows from   \cite{yos1}
     (Theorem 2.7, page 16)
     that  the functions  $g_\epsilon (\cdot,  x_0^* + y(\cdot))$
     and $h(y(\cdot))$ are
      almost periodic functions with values in $X$.
      Therefore,
       $\phi_\epsilon (t, y(t)) =
     h(y(t)) + g_\epsilon (t, y(t))$ is  also  almost periodic in $X$.
     By definition,
     given $\eta>0$,  there is a positive number $l$ (depending on $\eta$) such that
     every interval $I$ of length $l$   contains  a number $\sigma$ for which
      \be
     \label{p31_30}
     \| \phi_\epsilon (t+\sigma,  y(t+\sigma)) - \phi_\epsilon (t, y(t)) \|
     < \eta, \quad    \forall \  t \in \R.
     \ee
     Using \eqref{dichp}-\eqref{dichq} and \eqref{p31_30}, from \eqref{mapF} we obtain,
     for all $t \in \R$,
     $$
     \| {\cal{F}} (y) (t+\sigma) -{\cal{F}} (y) (t) \|_\alpha
     $$
     $$
     \le \|  \int_{-\infty}^{t+\sigma} e^{-A(t+\sigma -s)} (I-P) \phi_\epsilon (s, y(s)) ds
     -
     \int_{-\infty}^{t} e^{-A(t -s)} (I-P) \phi_\epsilon (s, y(s)) ds\|_\alpha
     $$
     $$
     +   \|  \int^{\infty}_{t+\sigma} e^{-A(t+\sigma -s)} P \phi_\epsilon (s, y(s)) ds
     -
     \int^{\infty}_{t} e^{-A(t -s)} P  \phi_\epsilon (s, y(s)) ds\|_\alpha
     $$
     $$
     \le  \|  \int_{-\infty}^{t} e^{-A(t -s)} (I-P) \left (
     \phi_\epsilon (s+\sigma,  y(s+\sigma))
     -  \phi_\epsilon (s,  y(s)) \right ) ds  \|_\alpha
     $$
     $$
     +   \|  \int^{\infty}_{t} e^{-A(t -s)} P
     \left (
     \phi_\epsilon (s+\sigma,  y(s+\sigma))
     -  \phi_\epsilon (s,  y(s)) \right ) ds  \|_\alpha
      $$
            $$
     \le M_1    \int_{-\infty}^{t} e^{-\beta(t -s)}
     \left ( 1 + (t-s)^{-\alpha}  \right )
     \|
     \phi_\epsilon (s+\sigma,  y(s+\sigma))
     -  \phi_\epsilon (s,  y(s)) \|  ds
     $$
     $$
     +  M_1  \int^{\infty}_{t} e^{\beta(t -s)}
    \|
     \phi_\epsilon (s+\sigma,  y(s+\sigma))
     -  \phi_\epsilon (s,  y(s))  \|  ds
      $$
  $$
     \le \eta M_1    \int_{-\infty}^{t} e^{-\beta(t -s)}
     \left ( 1 + (t-s)^{-\alpha}  \right ) ds
     + \eta  M_1  \int^{\infty}_{t} e^{\beta(t -s)} ds
     $$
     $$
     \le \eta M_1 \left (
     2\beta^{-1} + \beta^{\alpha -1} \Gamma (1-\alpha)
     \right ),
       $$
       which shows that ${\cal{F}} (y) : \R \to  X^\alpha$
       is almost periodic.
       By \eqref{p31_12} we see that
        $\cal{F}$ given by \eqref{mapF} maps $Z$ into itself.
        We next show that ${\cal{F}}: Z \to Z$ is a contraction.

        Let $y_1, y_2 \in Z$.   By  \eqref{mapF}  and \eqref{dichp}-\eqref{dichq}
        we have
        $$
        \| {\cal{F}} (y_1) (t) - {\cal{F}} (y_2) (t) \|_\alpha
        $$
        $$
        \le \| \int_{-\infty}^t e^{-A(t-s)} (I-P) (\phi_\epsilon (s, y_1(s))
        -\phi_\epsilon (s, y_2(s)) ) ds
        \|_\alpha
        $$
$$ +  \| \int^{\infty}_t e^{-A(t-s)} P  (\phi_\epsilon (s, y_1(s))
        -\phi_\epsilon (s, y_2(s)) ) ds
        \|_\alpha
        $$
        $$
        \le
   M_1  \int_{-\infty}^t e^{-\beta(t-s)}  (1 + (t-s)^{-\alpha} )
    \|  \phi_\epsilon (s, y_1(s))
        -\phi_\epsilon (s, y_2(s))  \|  ds
        $$
       \be
       \label{p31_40}
        + M_1  \int^{\infty}_t e^{\beta(t-s)}
    \|  \phi_\epsilon (s, y_1(s))
        -\phi_\epsilon (s, y_2(s))  \|  ds.
      \ee
      It follows from  \eqref{p31h1} and \eqref{p31g2} that
      $$
       \|  \phi_\epsilon (s, y_1(s))
        -\phi_\epsilon (s, y_2(s))  \|
         $$
         $$
         \le \| h(y_1(s)) -h(y_2(s)) \|
         + \| g_\epsilon (s, x_0^* + y_1(s)) - g_\epsilon (s, x_0^* + y_2(s)) \|
         $$
         $$
         \le \sup_{\| x\|_\alpha \le \delta} \| h^\prime (x) \|_{L(X^\alpha, X)} \| y_1(s) -y_2(s) \|_\alpha
         $$
         $$
         + \sup_{s \in \R}  \ \ \sup_{\|x \|_\alpha \le \|x_0^*\|_\alpha + 1}
         \| {\frac {\partial g_\epsilon}{\partial x}} (s,x) \|_{L(X^\alpha, X)} \| y_1(s)- y_2(s) \|_\alpha
         $$
         \be\label{p31_41}
         \le {\frac 1{M_1 (4\beta^{-1} + 2 \beta^{\alpha -1} \Gamma(1-\alpha) )} }\| y_1(s)- y_2(s) \|_\alpha.
         \ee
         By \eqref{p31_40} and \eqref{p31_41} we obtain, for all $t \in \R$,
         $$
        \| {\cal{F}} (y_1) (t) - {\cal{F}} (y_2) (t) \|_\alpha
        $$
        $$
       \le  {\frac 1{ (4\beta^{-1} + 2 \beta^{\alpha -1} \Gamma(1-\alpha) )} } \sup_{s \in \R} \| y_1(s)- y_2(s) \|_\alpha
       \left ( 2\int_{0}^\infty e^{-\beta s } ds   + \int_0^\infty e^{-\beta s} s^{-\alpha} ds
       \right )
       $$
        $$
        \le  {\frac 1{ (4\beta^{-1} + 2 \beta^{\alpha -1} \Gamma(1-\alpha) )} }
         (2\beta^{-1} +  \beta^{\alpha -1} \Gamma(1-\alpha ))
          \sup_{s \in \R} \| y_1(s)- y_2(s) \|_\alpha,
                 $$
                 which shows that
                 $$
        \sup_{t \in \R} \| ( {\cal{F}} (y_1)  - {\cal{F}} (y_2) )  (t) \|_\alpha
        \le {\frac 12}   \sup_{t \in \R} \| (y_1 - y_2 ) (t) \|_\alpha,        $$
        and hence ${\cal{F}}: Z \to  Z$ is a contraction.
        By the fixed  point  theorem, $\cal{F}$ has a unique fixed point $y_\epsilon^*$ in $Z$, which implies that
        $x_\epsilon^* = x_0^* + y^*_\epsilon$
        is the unique almost periodic solution of equation \eqref{nautoeq}
        satisfying
        $\sup\limits_{t\in \R} \| x_\epsilon ^* (t) -x_0^* \|_\alpha \le \delta$.
        Taking  $\delta =\delta_0$ we get the first part of  Lemma \ref{lem31}.
        For  arbitrary $\delta \in (0, \delta_0]$,
         the above process shows that
        $$
        \lim_{\epsilon \to 0} \sup_{t \in \R} \|  x_\epsilon^* (t) -x_0^* \|_\alpha =0,$$
        which completes the proof.
          \end{proof}

          We are now in a position to  prove our main results.

          {\bf Proof of Theorem 2.1}. \
          Given $\delta>0$ and $x_0 \in X^\alpha$,  denote by $\bar{B}_\alpha(x_0,\delta)$ the
          closed  ball in $X^\alpha$  with center $x_0$ and radius $\delta$, that is,
          $$ {\bar{B}}_\alpha (x_0, \delta) = \{  x \in X^\alpha: \  \|x-x_0 \|_\alpha \le \delta \}.
          $$
          Since equation \eqref{autoeq} has only  a set of finitely many equilibrium solutions
          $x^*_i$ ($1\le i \le n$), there is  $\delta_0>0$ such that
          ${\bar{B}}_\alpha (x_i^*, \delta_0) \bigcap  {\bar{B}}_\alpha (x_j^*, \delta_0)
          =\emptyset$ for all
            $i, j \in \{1, \cdots, n\}$ with $i\neq j$.
            Furthermore, for  each $i \in \{1, \cdots, n\}$, by Lemma \ref{lem31} there exist
            $\delta_i \in (0, \delta_0)$ and $\epsilon_i>0$ such that for every $\epsilon \in (0, \epsilon_i)$,
            equation \eqref{nautoeq} has a unique almost periodic solution
            $\phi_{i,\epsilon}^*$   for which
             $\phi_{i,\epsilon}^* (t)              \in {\bar{B}}_\alpha (x_i^*, \delta_i)$
             for all $t \in \R$.
            Let  $\epsilon_0 =\min \{\epsilon_i:  1 \le i \le n \}$.
            Then, for every $\epsilon \in (0, \epsilon_0)$,  we want to show that
            $\phi_{i, \epsilon}^*$,  $ 1\le i \le n$,   are the only almost periodic solutions
            of equation \eqref{nautoeq}.
            Suppose $x^*_\epsilon: \R \to X^\alpha$  is an arbitrary almost periodic solution
            of equation \eqref{nautoeq}. Since an almost periodic solution is a complete bounded solution,
            it follows from \cite{car1} that there is $i\in \{1, \cdots, n\}$ such that
            $$
            \lim_{t \to \infty} \| x^*_\epsilon (t) - \phi^*_{i,\epsilon}   (t)  \|_\alpha =0.
            $$
            Therefore, given $\eta>0$, there is $T>0$ such that for all $t \ge T$, the following holds:
            \be\label{pt1_1}
             \| x^*_\epsilon (t) - \phi^*_{i,\epsilon}   (t)  \|_\alpha          \le  \eta.
             \ee
             On the other hand,   since  $x^*_\epsilon$ is almost periodic, for the given $\eta>0$,
             there is   a positive number $l$ (depending on $\eta$) such that every interval $I$
             of length $l$ contains a number $s$ for which
             $$
             \|x^*_\epsilon (\tau  +s) - x^*_\epsilon (\tau) \|_\alpha \le \eta,
             \qquad \forall \   \tau  \in \R.
             $$
             This implies that for every $t \in \R$, there   is  a number
             $s_0 \in [T-t, \ T-t +l]$ such that
             \be
             \label{pt1_2}
                 \|x^*_\epsilon (t +s_0) - x^*_\epsilon (t) \|_\alpha \le \eta.
                 \ee
                 By  $s_0 \in   [T-t, \ T-t +l]$ we have
                 $t+s_0 \in [T, T+l]$ and hence   it follows from
                 \eqref{pt1_1}   that
                 \be\label{pt1_3}
             \| x^*_\epsilon (t+s_0) - \phi^*_{i,\epsilon}
              (t+s_0 )  \|_\alpha          \le  \eta.
             \ee
             By  \eqref{pt1_2} and \eqref{pt1_3} we obtain
             \be
             \label{pt1_5}
             \| x^*_\epsilon (t) - \phi^*_{i,\epsilon}
              (t+s_0 )  \|_\alpha          \le    2\eta.
             \ee
             Since $\phi^*_{i,\epsilon} (t)  \in {\bar{B}}_\alpha (x_i^*, \delta_i )$
             for all $t \in \R$, we find from \eqref{pt1_5}
              that, for all $t \in \R$,
             $$
              \| x^*_\epsilon (t) -  x^*_i \|_\alpha
                     \le   \delta_i +   2\eta,
                     \quad \forall \ \eta>0.
                     $$
                     Taking $\eta \to   0 $, we
                     see    that
                     $x^*_\epsilon $ is an almost periodic solution of equation
                     \eqref{nautoeq} which  satisfies
                     $x^*_\epsilon (t) \in {\bar{B}}_\alpha(x_i^*, \delta_i)$
                     for all $t \in \R$.
                     Note that  $\phi^*_{i, \epsilon}$ is the unique almost periodic solution
                     of the equation which  belongs to
        $ {\bar{B}}_\alpha(x_i^*, \delta_i)$. Therefore we must have
        $
        x^*_\epsilon (t) = \phi^*_{i,\epsilon} (t) $ for all $  \ t \in \R$.
        Since  $  x^*_\epsilon  $ is an arbitrary almost periodic solution,
        we see that
         the non-autonomous equation
        has no any other almost  periodic  solutions except
                  $\phi^*_{i, \epsilon}$, $1\le i \le n$.
                          This  concludes the proof of Theorem 2.1.

        As mentioned  earlier,  Corollary 2.2 was proved by Carvalho et. al. in \cite{car1}
        when $\phi^*_{i, \epsilon}$,  $1\le i \le n$, are complete bounded solutions. In the present
        paper, we  demonstrate  that these solutions are actually almost periodic solutions
        under almost periodic solutions. Therefore, Corollary 2.2 is an immediate consequence
        of Theorem 2.1 and the results of \cite{car1}.

\section{Applications}
\setcounter{equation}{0}

In this section, we discuss an application of our results to the
Chafee-Infante  equation and describe the almost periodic dynamics of
the non-autonomous equation under small perturbations.
The one-dimensional   Chafee-Infante  equation reads
\be
\label{cfeq1}
{\frac {\partial u}{\partial t}} -{\frac {\partial ^2 u}{\partial x^2}}
=\lambda (u - u^3), \quad x \in (0, \pi), \quad t>0,
\ee
with the boundary condition
\be
\label{cfeq2}
u(t, 0) = u(t, \pi) =0, \quad  t>0,\ee
and the initial condition
\be
\label{cfeq3}
u(0, x) =u_0(x), \quad x \in (0, \pi),\ee
where $\lambda$ is a positive parameter.

Let $A_0  = -\partial_{xx}$ with domain
$D(A_0) =H^2((0,\pi))\bigcap H^1_0((0,\pi))$.
Then  $A_0$   is a  sectorial operator
in $X=L^2((0,\pi))$
and the eigenvalues of $A_0$ are given by
$ \lambda_n = n^2$ where $n$ is any positive  integer.
It is well known that problem \eqref{cfeq1}-\eqref{cfeq3}
is well-posed in $D(A_0^{\frac 12}) = H^1_0((0,\pi))$.
In other words,
system  \eqref{cfeq1}-\eqref{cfeq3} defines a
  continuous  semigroup
$\{S(t)\}_{t \in \R}$ on $H^1_0((0,\pi))$
(see, e.g., \cite{hen1}).
This semigroup    has a global attractor $\cal{A}$  in
$H^1_0((0,\pi))$.
The structure of equilibrium solutions of problem  \eqref{cfeq1}-\eqref{cfeq3}
is well understood.  Actually,  for every $\lambda \in (n^2, (n+1)^2)$
 where $n$ is any nonnegative
integer,  it was proved by  Chafee and Infante in \cite{cha1} that  problem \eqref{cfeq1}-\eqref{cfeq3}
has exactly $2n+1$ equilibrium solutions.
It was further proved by Henry in \cite{hen2} that all these equilibrium solutions are
hyperbolic.
Let $V:  H^1_0((0,\pi)) \to \R$ be given  by
$$
V(u) = \int_0^\pi  \left (
 ({\frac {\partial u}{\partial x}})^2  - \lambda u^2
+{\frac 12} \lambda u^4
\right ) dx,  \quad
\forall \ u \in H^1_0((0,\pi)).
$$
Then $V$ is a Liapunov function of $\{S(t)\}_{t \in \R}$.
So assumptions   (H1)-(H3) are  all fulfilled in this case,
and the attractor $\cal{A}$ is given by the union of unstable manifolds
of the $(2n+1)$
 equilibrium solutions. Further the Hausdorff dimension of $\cal{A}$
is $n$ as shown in \cite{hen1}.

Suppose $g: \R \to \R$ is an almost periodic  function and $h \in
L^2((0,\pi))$. Given $\epsilon>0$, consider now the non-autonomously
perturbed equation, for  all $\tau \in \R$,
 \be
\label{ncfeq1}
{\frac {\partial u}{\partial t}} -{\frac {\partial ^2 u}{\partial x^2}}
=\lambda (u - u^3) +\epsilon g(t) h(x), \quad x \in (0, \pi), \quad t>\tau,
\ee
with the boundary condition
\be
\label{ncfeq2}
u(t, 0) = u(t, \pi) =0, \quad  t> \tau,\ee
and the initial condition
\be
\label{ncfeq3}
u(\tau, x) =u_0(x), \quad x \in (0, \pi),
\ee
Set $g_\epsilon (t,x) = \epsilon g(t) h(x)$
for all $t \in \R$  and $x \in (0,\pi)$.
Then it is evident that $g_\epsilon$ satisfies condition
(H4).  Given $\epsilon>0$,  the existence of   pullback attractor
$\{{\cal{A}}_\epsilon (t) \}_{t \in \R}$
 for problem
\eqref{ncfeq1}-\eqref{ncfeq3}  in $H^1_0((0,\pi))$
can be proved by standard arguments, see, e.g., \cite{che1, wan1}.
Note that $g: \R \to \R$ is bounded since it is almost periodic.
Then it is easy to verify that  the attractor
$\{{\cal{A}}_\epsilon (t) \}_{t \in \R}$
indeed satisfies condition (H5)
with $\epsilon_0 =1$.
Applying Theorem 2.1 and Corollary 2.2 to system
\eqref{ncfeq1}-\eqref{ncfeq3} we have
\begin{thm}
\label{thmappl}
Let $n$ be a nonnegative integer and $\lambda \in (n^2, (n+1)^2)$.
Then There is $\epsilon_0>0$ such that for every
$\epsilon \in (0, \epsilon_0)$,  problem
\eqref{ncfeq1}-\eqref{ncfeq3}
has exactly  $2n+1$ almost periodic solutions
$\phi^*_{i,\epsilon}: \R \to  H^1_0((0,\pi))$, $1\le i \le 2n+1$.
The system has an  $n$-dimensional  pullback attractor
$\{{\cal{A}}_\epsilon (t) \}_{t \in \R}$
which is the union of the unstable manifolds of
the almost periodic solutions.
Further, every solution of problem
\eqref{ncfeq1}-\eqref{ncfeq3}
converges to one of those almost periodic  solutions.
\end{thm}

\end{document}